\def\draft{n}
\documentclass{amsart}[12pt]

\usepackage{latexsym,amssymb,graphicx,amscd,amsmath}

\def\printname#1{
	\if\draft y
		\smash{\makebox[0pt]{\hspace{-0.5in}
			\raisebox{8pt}{\tt\tiny #1}}}
	\fi
}

\def\lbl#1{\label{#1}\printname{#1}}

\newtheorem{thm}{Theorem}[section]
\newtheorem{lem}[thm]{Lemma}
\newtheorem{prop}[thm]{Proposition}

\newtheorem{rem}[thm]{Remark}

\newcommand{\BZ}{{\mathbb{Z}}}

\newcommand{\BC}{{\mathbb{C}}}

\DeclareMathOperator{\Id}{Id}

\DeclareMathOperator{\PGL}{PGL}

\begin{document}

\title{On powers of half-twists in $M(0,2n)$}

\author{Gregor Masbaum}
\address{Institut de Math{\'e}matiques de Jussieu (UMR 7586 du CNRS)\\
Case 247\\
4 pl. Jussieu\\
75252 Paris Cedex 5\\
FRANCE }
\email{gregor.masbaum@imj-prg.fr}
\urladdr{webusers.imj-prg.fr/~gregor.masbaum}


\thanks{Research supported in part by the center of excellence grant ``Center for Quantum Geometry of Moduli Spaces (QGM)'' DNRF95 from the Danish National Research Foundation.}

\date{October 13, 2016}

\begin{abstract} We use elementary skein theory to prove a version of a result of Stylianakis \cite{S} who showed that under mild restrictions on $m$ and $n$, the normal closure of the \hbox{$m$-th} power of a half-twist has infinite index in    the mapping class group of a sphere with $2n$ punctures.
\end{abstract}

\maketitle
%\tableofcontents

\section{Introduction} \lbl{sec.intro}

Let $M(0,2n)$ be the mapping class group of the $2$-sphere $S^2$ fixing (setwise) a set of $2n$ points $p_1, \ldots, p_{2n} \in S^2$. It is well-known \cite{Bi} that $M(0,2n)$ is a quotient of the braid group $B_{2n}$ on $2n$ strands, where the braid generator $\sigma_i$  ($i=1, \ldots, 2n-1$) maps to the mapping class $h_i \in M(0,2n)$ which is a {\em half-twist} permuting $p_i$ and $p_{i+1}$ and fixing all other points $p_j$. Stylianakis recently showed the following:

\begin{thm}[Stylianakis \cite{S}] \lbl{1.1} For $2n \geq 6$ and $m\geq 5$, the normal closure of $h_i^m$ has infinite index in $M(0,2n)$.
\end{thm}  
 (Note that the normal closure does not depend on $i$, as the $h_i$ are all conjugate.)

For $2n=6$, this result was known and is due to Humphries \cite{H}, as it is equivalent (by the Birman-Hilden Theorem) to Humphries' result \cite[Thm. 4]{H} that the normal closure of the $m$-th power of a non-separating Dehn twist has infinite order in the genus $2$ mapping class group for $m\geq 5$. Humphries' method was to employ the Jones representation \cite{J} of the genus $2$ mapping class group together with an explicit computation. Stylianakis' generalization proceeds by using certain Jones representations of $M(0,2n)$, but his proof involves some non-trivial representation theory.

 In this paper, we give an elementary skein-theoretic proof of the following:

\begin{thm} \lbl{1.2} For $2n \geq 4$ and $m\geq 6$, the normal closure of $h_i^m$ has infinite index in $M(0,2n)$.
\end{thm}

The key point in the proof of Theorem~\ref{1.2} is a simple $2\times 2$ matrix calculation that I essentially did in \cite{Madeira}. Note that Theorem~\ref{1.2} implies Stylianakis' result for $m\geq 6$. Theorem~\ref{1.2} does not hold when $m=5$ and  $2n=4$, as $M(0,4)/(h_i^5=1)$ is a finite group (the alternating group $A_5$).   I believe that the remaining case ($m=5$, $2n\geq 6$) of Stylianakis' theorem can also be proved using the skein-theoretic method exposed below, but it would require a calculation with $5\times 5$ matrices which I have not done (see Remark~\ref{remaining}).

\section{Strategy of the proof}

The proof  will be based on the  representation of the braid group $B_{2n}$ on  the Kauffman bracket  \cite{K} skein module 
%%$\BS(S^2, (1)_{2n})$ 
of the $3$-ball relative to $2n$ marked points on the boundary. We will show that for an appropriate choice of Kauffman's skein variable $A$, this representation induces a projective-linear representation  $$\rho : M(0,2n) \rightarrow \PGL_d(\BC)$$ (where $d$ depends on $n$) so that 
\begin{enumerate}
\item[(i)] $\rho(h_i^m)=1$, and
\item[(ii)] the image $\rho(M(0,2n))$ is an infinite group. 
\end{enumerate} Clearly this will imply that the normal closure of $h_i^m$ has infinite index in $M(0,2n)$.

\begin{rem}{\em Stylianakis used the same strategy applied to a certain Jones representation of $M(0,2n)$.  Actually, up to normalization and change of variables, the representation $\rho$ is equivalent to the Jones representation for the rectangular Young diagram with $2$ rows of length $n$. (We shall not make use of this fact in this paper.) For the purpose at hand, I find the skein-theoretic approach much easier.}\end{rem}
  
\begin{rem}{\em Funar \cite{F} showed that the normal closure of the $m$-th power of a Dehn twist has infinite index in the mapping class group of a genus $g$ surface (with some restrictions on $m$ and $g$) using the above strategy applied to TQFT-representations of mapping class groups.  Our representation $\rho$ can also be viewed as a TQFT representation of $M(0,2n)$. But for us, TQFT is not actually needed. We shall only need Birman's presentation \cite[Thm.~4.5]{Bi} of $M(0,2n)$ as a quotient of $B_{2n}$ and elementary skein theory.
}\end{rem}

\section{Proof of Theorem~\ref{1.2}}

We start with the  representation of the braid group $B_{2n}$ on  the Kauffman bracket skein module 
%%$\BS(S^2, (1)_{2n})$ 
of the $3$-ball relative to $2n$ marked points on the boundary. Let us recall how this representation, which we denote by $\rho$, is defined. The skein module is a free $\BZ[A,A^{-1}]$-module of dimension $$d=\frac 1 {n+1} {2n \choose n}$$ (the Catalan number). Its elements are represented by $\BZ[A,A^{-1}]$-linear combinations of  $(0,2n)$-tangle diagrams, that is, tangle diagrams in a rectangle relative to $2n$ marked points at the top of the rectangle. The diagrams are considered modulo the Kauffman skein relations (which will be stated shortly).  The skein  module has a standard basis given by tangle diagrams without crossings and without closed circles. For example, if the number of points is $2n=4$, the dimension is $d=2$ and the basis is given by the two diagrams 
$$D_1 =\ \ \ \begin{minipage}{0.2in}\includegraphics[width=0.4in,height=0.16in]{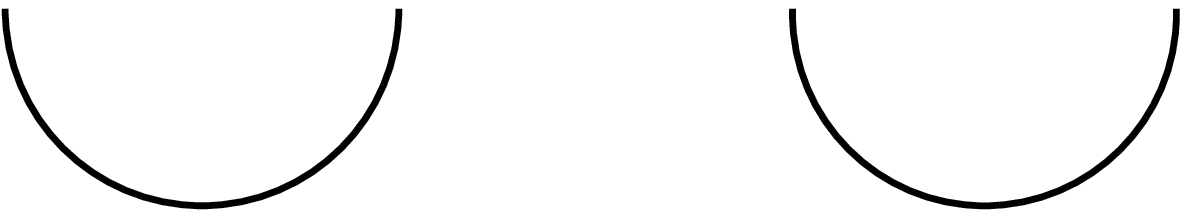}\end{minipage} \hskip 2cm  D_2= \ \ \ \begin{minipage}{0.2in} \includegraphics[width=0.4in,height=0.16in]{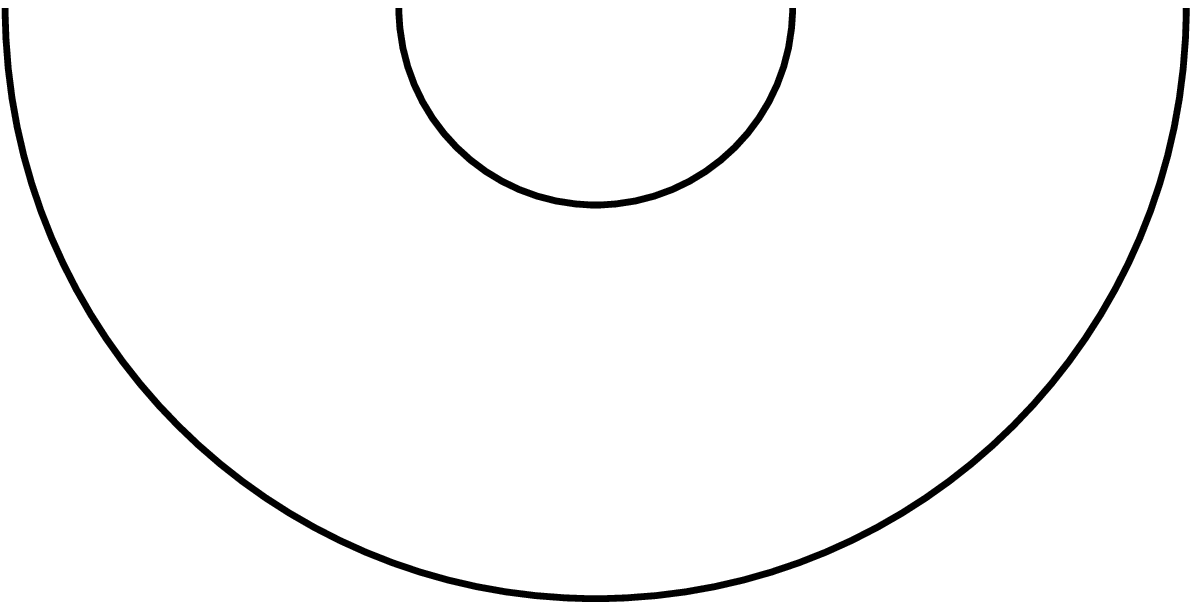}\end{minipage}$$ Below we specialize $A$ to a non-zero complex number, so that  the skein module with this basis (ordered in some arbitrary fashion) is identified
 with $\BC^d$. 

 The $i$-th braid generator $\sigma_i$ acts on a diagram $D$ by gluing the usual braid diagram of $\sigma_i^{-1}$ on top of $D$ (that is, the braid diagram  which has a crossing  $\includegraphics[width=0.1in]{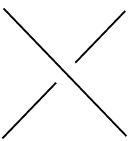}$ at the $i$-th and $(i+1)$-st strand and all other strands are vertical). (We use inverses here so as to get a left action of $B_{2n}$ on the skein module.) The Kauffman bracket skein relation 
\[ 
\begin{minipage}{0.2in}\includegraphics[width=0.2in]{L+2.eps}\end{minipage}
=  
 \ \ A \  \ 
\begin{minipage}{0.2in}\includegraphics[width=0.2in]{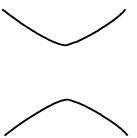} \end{minipage}
+
\ \  A^{-1} \  \ \begin{minipage}{0.2in} \includegraphics[width=0.2in]{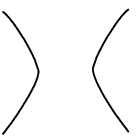}  \end{minipage}
   \]
implies that $$\rho(\sigma_i) =  A \,  \rho (E_i) + A^{-1} \Id ~,$$ where $E_i$ has $\includegraphics[width=0.1in]{L0.eps}$ at the appropriate place and all other strands are vertical.  The second Kauffman skein relation, which fixes the value of an unknot diagram to $-A^2-A^{-2}$,  implies that $$\rho (E_i)^2 = (-A^2-A^{-2})\,  \rho (E_i)~.$$ A simple recursion now establishes that $$ \rho(\sigma_i^m)= P_m(A) \,  \rho (E_i) + A^{-m} \Id ~,$$ where $P_m(A)= A^{2-m} (1 - A^4 + A^8 - \ldots +(-1)^{m-1} A^{4m-4}$). Thus we have the following

\begin{prop}\lbl{3.1} If $A\in \BC$ satisfies $P_m(A)=0$, then $\rho(\sigma_i^m)=A^{-m}\Id$ is the identity element in $\PGL_d(\BC)$.
\end{prop}

From now on, we assume that $A$ is a zero of the polynomial $P_m(A)$. Note that all zeros of $P_m(A)$ are roots of unity. We shall make a precise choice of $A$ later.

\begin{prop}\lbl{3.2} For any $A\in \BC^*$, the homomorphism   $\rho: B_{2n} \rightarrow \PGL_d(\BC)$  factors through $M(0,2n)$.
\end{prop}
\begin{proof}  This is well-known but here is a proof. The group $M(0,2n)$ is the quotient of $B_{2n}$ by the relations $R_1=R_2=1$ where $$R_1= \sigma_1 \sigma_2 \cdots \sigma_{2n-1}\sigma_{2n-1} \sigma_{2n-2}\cdots \sigma_1$$ $$R_2= (\sigma_1 \sigma_2 \cdots \sigma_{2n-1})^{2n}$$ (see \cite[Thm~4.5]{Bi}). Using the isotopy invariance of the Kauffman bracket, it is easy to check that $$\rho(R_1)= (-A^3)^2 \Id$$ $$ \rho(R_2) = (-A^3)^{2n} \Id$$ (see \cite[\S 1.3]{Sa}). This proves the proposition.\end{proof}

\begin{rem}{\em For appropriate roots of unity $A$, the induced projective-linear representation of $M(0,2n)$ is a TQFT representation, as follows from the skein-theoretic construction of Witten-Reshetikhin-Turaev TQFT in \cite{BHMV}.
}\end{rem}

By abuse of notation, we denote the induced homomorphism $M(0,2n) \rightarrow \PGL_d(\BC)$, which sends $h_i$ to $\rho(\sigma_i)$, again by $\rho$.  Thus we have realized condition (i) of the strategy outlined in \S 2. To realize condition (ii), it suffices to find an element $\phi\in B_{2n}$ so that $\rho(\phi)$ has infinite order in $\PGL_d(\BC).$ We now show that $\phi = \sigma_1^{2}\sigma_2^{-2}$ works.

Recall the diagrams $D_i$ ($i=1,2$) depicted above.  By taking disjoint union of $D_i$ with some fixed $(0,2n-4)$-tangle diagram $\widetilde D$ (so that the first $4$ points are the boundary of $D_i$, and the remaining $2n-4$ points are the boundary of $\widetilde D$), we get two diagrams $D'_1$ and $D'_2$ which form part of a basis of our skein module. The two-dimensional subspace spanned by $D'_1$ and $D'_2$ is preserved by both $\rho(\sigma_1)$ and $\rho(\sigma_2)$. On this subspace, $\rho(\sigma_1)$ and $\rho(\sigma_2)$ act by the following matrices:
$$\rho(\sigma_1) = \left[ \begin{array}{cc}
-A^3 & A \\
0 & A^{-1} \end{array} \right] \hskip 1cm 
\rho(\sigma_2) = \left[ \begin{array}{cc}
A^{-1} & 0 \\
A & -A^3 \end{array} \right]$$ (This follows immediately from the Kauffman relations.) 
A straightforward calculation now gives that the matrix of $\rho(\sigma_1^{2}\sigma_2^{-2})$ acting on this $2$-dimensional subspace is 
$$M= \left[ \begin{array}{cc}
 2-A^4-A^{-4} +A^8 & -A^{-2}+A^{-6} \\
A^{-2} - A^{-6} &  A^{-8}\end{array} \right]$$

Clearly, if $M$ has infinite order in $\PGL_2(\BC)$, then $\rho(\sigma_1^{2}\sigma_2^{-2})$ has infinite order in  $\PGL_d(\BC)$.

\begin{lem}\lbl{3.3} $M$ has infinite order in $\PGL_2(\BC)$ provided the order $r$ of the root of unity $q=A^4$ satisfies $r\geq 3$ and $r\not\in\{4,6,10\}$. 
\end{lem} 

\begin{proof} For $r\geq 5$ and $r\not\in\{6,10\}$, this is shown in \cite{Madeira}, as one can check that the matrix $M$ is conjugate to the one computed in \cite{Madeira}. We can also apply the argument of \cite{Madeira} directly to our matrix, as follows. Note that $M$ 
has determinant $1$ and trace 
$$t=2-q-q^{-1}+q^2+q^{-2}$$ where $q=A^4$. If $M$ has finite order in $\PGL_2(\BC)$, then its eigenvalues $\lambda$ and $\lambda^{-1}$ must satisfy $\lambda^N=\lambda^{-N}$ for some $N$, so $\lambda$ is a root of unity. But this is impossible, as we can find a primitive $r$-th root $q\in \BC$ such that $|t|=|\lambda + \lambda^{-1}|>2$ (see \cite{Madeira}). Thus $M$ has infinite order in $\PGL_2(\BC)$.

In the remaining case $r=3$, it suffices to observe that in this case we have  $t=2$, so $M$ is conjugate to $$\left[ \begin{array}{cc}
1 & c \\
0 & 1 \end{array} \right]$$ with $c\neq 0$ (since $M$ is not the identity matrix). 
\end{proof}

The proof of Theorem~\ref{1.2} is now completed as follows. For $m\geq 6$, we choose $A$ to be a primitive $N$-th root of unity, as follows:
\begin{enumerate}
\item[$\bullet$] For $m=6$, we take $N=12$.
 \item[$\bullet$] For $m=10$, we take $N=20$.
\item[$\bullet$] For odd $m\geq 7$, we take $N=8m$.
\item[$\bullet$] For even $m\geq 8$, $m \neq 10$, we take $N=4m$.
\end{enumerate}

Then $P_m(A)=0$, so Prop.~\ref{3.1} applies. Also  $q=A^4$ has order $r\geq 3$,  $r\not\in\{4,6,10\}$, so Lemma~\ref{3.3} applies. Thus $\rho$ satisfies condition (i) because of Prop.~\ref{3.1}, and $\rho$ satisfies condition (ii) because the  matrix $\rho(\sigma_1^{2}\sigma_2^{-2})$ has infinite order in $\PGL_d(\BC)$.
This completes the proof.

\begin{rem}\lbl{remaining}{\em I expect that the remaining case ($m=5$, $2n \geq 6$) of Stylianakis' theorem (see Theorem~\ref{1.1}) can also be proved using the skein-theoretic representation $\rho$ evaluated at a root of unity $A$ so that $P_5(A)=0$. It suffices to find $\phi\in B_6$ so that the $5\times 5$ matrix $\rho(\phi)$ has infinite order in $\PGL_5(\BC)$. This will imply the result for $M(0,2n)$  with $2n \geq 6$ for the same reason as above.  Stylianakis describes such an element $\phi$ and shows that it has infinite order in the Jones representation he uses. Actually $\phi$ is closely related to the element originally used by Humphries \cite{H}.   Note that {\em modulo} identifying our skein-theoretic representation of $M(0,6)$ with the Jones representation used by Humphries, the fact that $\rho(\phi)$ has infinite order is already shown by  Humphries. There seems to be no advantage in redoing the relevant $5\times 5$ matrix computation directly from the skein-theoretic approach, and I have not attempted to do so. 
}\end{rem}

\begin{rem}{\em  The proof of Prop.~\ref{3.2} shows that one can rescale $\rho$ to get a representation $\hat \rho$ of $B_{2n}$ which descends to $M(0,2n)$ as a {\em linear} representation: put $$\hat \rho(\sigma_i) = \theta^{-1}  \rho(\sigma_i)$$ where $\theta^{4n-2} = (-A^3)^{2}=A^{6}$; then $\hat\rho(R_1)=\hat\rho(R_2)=\Id.$  Note that  $$\hat\rho(\sigma_i^m)=(\theta A)^{-m} \Id~. $$ One may wonder whether $\theta$ can be chosen so that $(\theta A)^{-m}=1$. In general, the answer is no. For example, if $m$ is odd, then $P_m(A)=0$ implies $A^{4m}=-1$, and one computes (using $\theta^{4n-2} =A^{6}$) that 
$$((\theta A)^{-m})^{4n-2}= A^{-4m(n+1)}=(-1)^{n+1}~.$$ Thus $(\theta A)^{-m}\neq 1$  if $m$ is odd and $n$ is even.
}\end{rem} 

%\noindent{\bf Acknowledgment.} I thank the Centre for Quantum Geometry of Moduli Spaces (QGM), Aarhus, Denmark,  for hospitality while this note was written.

\end{document}